\documentclass[11pt]{amsart}
\usepackage{amssymb,amsmath,amsthm}
\usepackage[margin=1.15in]{geometry}
\usepackage{enumerate}
\usepackage{booktabs}
\usepackage[dvipsnames]{xcolor}
\usepackage[colorlinks=true,linkcolor=blue!60!black,citecolor=green!40!black]{hyperref}

\newtheorem{theorem}{Theorem}[section]
\newtheorem{thmx}{Theorem}

\newtheorem{corx}[thmx]{Corollary}
\newtheorem{lemma}[theorem]{Lemma}
\newtheorem{proposition}[theorem]{Proposition}
\newtheorem{corollary}[theorem]{Corollary}
\theoremstyle{definition}

\newtheorem{example}[theorem]{Example}
\theoremstyle{remark}
\newtheorem{remark}[theorem]{Remark}

\newcommand{\Ric}{\operatorname{Ric}}
\newcommand{\Hess}{\operatorname{Hess}}
\newcommand{\Vol}{\operatorname{Vol}}
\newcommand{\Y}{\mathcal{Y}}
\newcommand{\Nb}{\overline{N}}
\newcommand{\Mb}{\overline{M}}
\newcommand{\dN}{\partial N}
\newcommand{\E}{\mathcal{E}}

\title[Non-collapsing of warped fibers]{Non-collapsing of fibers in warped-product\\ Poincar\'e--Einstein manifolds}
\author{Mohammad Javaheri}
\address{Department of Mathematics and Statistics, Siena University\\
515 Loudon Road, Loudonville, NY 12211}
\email{mjavaheri@siena.edu}

\begin{document}

\begin{abstract}
Let $g=g_M+u^2h$ be a Poincar\'e--Einstein (asymptotically hyperbolic Einstein, AHE) metric of class $C^{3,\alpha}$
on $N^{n+1}=M^q\times F^p$, where $(F,h)$ is compact with unit volume and $\Ric_h=\lambda h$, $\lambda\le 0$.
We show that if the conformal infinity $[\gamma]$ has positive Yamabe constant, then
\[
u\;\ge\;\Theta[\gamma]:=\sup_{\eta}\Big(\min_{\dN}\psi_\eta\Big)\Big(\frac{\min_{\dN} s_\eta}{n(n-1)}\Big)^{1/2}>0,
\]
the supremum being over representatives $\eta\in[\gamma]$ of positive scalar curvature and $\psi_\eta$ the
boundary warping function of $\eta$. In particular $u\ge (\Y[\gamma]/n(n-1))^{1/2}\min\psi_{\eta_Y}$ for the
Yamabe metric $\eta_Y$. The estimate rests on a comparison principle between $\ln u$ and Lee's
eigenfunction, which also applies to each factor of a multiply warped product, and
pointwise yields $u\ge(\min\psi_\eta)\,w_\eta$ when $\Y[\gamma]\ge 0$. We give an explicit bound in terms of the
product representative of $[\gamma]$, show that the hypotheses force $q\ge 3$ and $\partial M$ connected,
prove lower semicontinuity of $\Theta$ in the boundary data, and discuss examples (hyperbolic Fermi coordinates,
AdS--Schwarzschild) showing that the hypothesis $\lambda\le0$ cannot be removed. 
\end{abstract}

\maketitle

\section{Introduction}

\subsection{Motivation}
A \emph{Poincar\'e--Einstein} (conformally compact Einstein) metric is a complete metric $g$ with
$\Ric_g=-ng$ on the interior of a compact manifold $\Nb^{n+1}$, such that $\rho^2g$ extends to $\Nb$ for a
defining function $\rho$ of $\dN$; its conformal infinity is $[\rho^2g|_{\dN}]$. The central question is the
\emph{filling problem}: given a conformal class $[\gamma]$ on a closed manifold $X^n$, describe the
Poincar\'e--Einstein metrics with conformal infinity $[\gamma]$. Graham--Lee~\cite{GL} proved existence for
$[\gamma]$ near the round class on $S^n$, and Anderson~\cite{An1,An2} developed a global degree theory for
the boundary map $\Pi:g\mapsto[\gamma]$ in dimension four, whose key analytic input is that $\Pi$ be
\emph{proper}: a sequence of fillings whose boundary data converge should subconverge.

Properness has since been established in dimension four under conformally invariant hypotheses by
Chang--Ge~\cite{CG1} and Chang--Ge--Qing~\cite{CGQ}. Their arguments are blow-up analyses: one rescales at
points where curvature concentrates or the injectivity radius degenerates, and shows that the only possible
limits are rigid (hyperbolic space). Chang--Ge--Qing in particular obtain injectivity radius estimates that
remove earlier assumptions on Yamabe constants, and as an application prove global uniqueness of the
Graham--Lee fillings near the round $S^3$. Very recently Chang--Ge~\cite{CG2} combined such compactness
with the continuity method to produce fillings of boundary metrics near the canonical ones on $S^3$ and on
$S^1\times S^2$. In a different direction, Li--Qing--Shi~\cite{LQS} showed that a large Yamabe constant of
the conformal infinity forces curvature pinching and rigidity in the interior. In all of these, the Yamabe
constant of the conformal infinity is the quantity that controls interior geometry, and what must be ruled
out is \emph{degeneration}: curvature blow-up or collapse of part of the interior.

The results above are four-dimensional, and their blow-up arguments do not directly give
explicit constants. This note studies the same mechanism in arbitrary dimension within the class of
\emph{warped products} $g=g_M+u^2h$, where one degeneration mode can be isolated and controlled
explicitly. Warped products are the natural symmetric test class: the Euclidean AdS black holes, the AdS
solitons, the static metrics with circle actions of Anderson--Chru\'sciel--Delay~\cite{ACD}, and the
examples foliated by products of Einstein manifolds and circle bundles constructed by D.~Chen~\cite{Ch} are
all of this type. The structure of warped-product Einstein metrics is well understood in the complete
setting: He--Petersen--Wylie~\cite{HPW} showed that such a metric is determined by its base, fiber type and
Einstein constant, but that in the non-compact case the warping function need not be unique, and explicit
warping functions over hyperbolic bases have recently been computed in~\cite{SMS}. What is missing is control
of the warping function \emph{from the conformal infinity}. For warped products the basic mode of
degeneration is the \emph{collapse of the fibers}, $\inf_Mu\to0$, which would show up in any compactness
argument as injectivity radius collapse along the fibers.

Our main result rules out this collapse, with an explicit constant, whenever the fibers are Einstein with
nonpositive Ricci curvature and the conformal infinity has positive Yamabe constant. This is the same
hypothesis under which Witten--Yau~\cite{WY} proved connectedness of $\partial N$, Lee~\cite{Lee} computed the
bottom of the spectrum, and Li--Qing--Shi obtain their estimates; our proof uses Lee's eigenfunction.
The hypothesis on the fibers cannot be dropped: we show that small Euclidean AdS--Schwarzschild black holes,
whose fibers are round spheres, violate the estimate (Example~\ref{ex:schw}). Their conformal infinity is
$S^1\times S^{n-1}$, which for $n=3$ is precisely the boundary considered in~\cite{CG2}.

\subsection{Setting}
Let $\Nb=\Mb\times F$ be a compact $(n+1)$-manifold with boundary $\dN=\partial M\times F$, where $\Mb^q$ is
compact with boundary and $F^p$ is closed and connected, $p+q=n+1$. We consider metrics
\begin{equation}\label{eq:wp}
g=g_M+u^2h,\qquad \Ric_g=-n\,g,
\end{equation}
with $g_M$ complete on $M$, $u\in C^\infty(M)$ positive and $h$ a metric on $F$ normalized by
$\Vol(F,h)=1$. We assume $g$ is AH of class $C^{3,\alpha}$ (Section~\ref{sec:prelim}). By
Lemma~\ref{lem:warped} below, $(F,h)$ is then automatically Einstein, $\Ric_h=\lambda h$ with $\lambda$
constant. Let $\E^{3,\alpha}_W(N)$ denote the set of such metrics with $\lambda\le 0$, and $\Pi_W$ the
restriction of the boundary map to $\E^{3,\alpha}_W(N)$.

For a representative $\eta$ of the conformal infinity, let $\rho_\eta$ be the associated geodesic
defining function and define the \emph{boundary warping function}
\begin{equation}\label{eq:psi}
\psi_\eta:=\lim \rho_\eta u\;\in C^0(\dN),\qquad \psi_\eta=\Big(\frac{\eta(V,V)}{h(V,V)}\Big)^{1/2}\ \ (0\ne V\in TF).
\end{equation}
It depends only on $\eta$ (and $h$), is positive, and satisfies $\psi_{f^2\eta}=f\psi_\eta$. If $\eta$ is
$F$-invariant then $\eta=\hat g+\psi_\eta^2h$ is itself a warped product on $\partial M\times F$.

\subsection{Main results}
Throughout, $s_\eta$ is the scalar curvature of $\eta$, $\Y[\gamma]$ the Yamabe constant, and for
$\eta\in[\gamma]$ with $s_\eta>0$ we set
\[
\Theta_\eta:=\Big(\min_{\dN}\psi_\eta\Big)\Big(\frac{\min_{\dN}s_\eta}{n(n-1)}\Big)^{1/2},\qquad
\Theta[\gamma]:=\sup\{\Theta_\eta:\ \eta\in[\gamma],\ s_\eta>0\}.
\]
$\Theta_\eta$ is invariant under constant rescaling of $\eta$.

\begin{thmx}\label{thm:A}
Let $g=g_M+u^2h\in\E^{3,\alpha}_W(N)$ have conformal infinity $[\gamma]$, $n\ge3$.
\begin{enumerate}[\rm(i)]
\item If $\eta\in[\gamma]$ has $s_\eta\ge0$ and $w_\eta$ is Lee's eigenfunction for $\eta$ (Lemma~\ref{lem:lee}),
then $u\ge(\min_{\dN}\psi_\eta)\,w_\eta$ on $N$.
\item If $s_\eta>0$, then $u>\Theta_\eta$ on $N$.
\item If $\Y[\gamma]>0$, then $u\ge\Theta[\gamma]>0$ on $N$. In particular, if $\eta_Y$ is a unit-volume Yamabe
metric in $[\gamma]$,
\[
u\;>\;\Big(\frac{\Y[\gamma]}{n(n-1)}\Big)^{1/2}\min_{\dN}\psi_{\eta_Y}.
\]
\end{enumerate}
\end{thmx}

A particularly explicit choice of $\eta$ is
the product representative:

\begin{corx}\label{cor:B}
Let $\gamma=\hat g+\psi^2h$ be an $F$-invariant representative of the conformal infinity of
$g\in\E^{3,\alpha}_W(N)$, and let $\bar g:=\psi^{-2}\hat g$ on $\partial M$ (it depends only on $[\gamma]$). If
\[
\sigma:=\min_{\partial M}s_{\bar g}+p\lambda>0,\qquad\text{then}\qquad u>\Big(\frac{\sigma}{n(n-1)}\Big)^{1/2}\ \text{on }N.
\]
\end{corx}

The comparison principle behind Theorem~\ref{thm:A} uses only the identity
$\Delta_g\ln u=n+\lambda u^{-2}$, which we show persists for each factor of a multiply warped product:

\begin{thmx}\label{thm:C}
Let $g=g_M+\sum_{i=1}^k u_i^2h_i$ be AHE of class $C^{3,\alpha}$ on $M\times F_1\times\dots\times F_k$, with
$\Vol(F_i,h_i)=1$. Then $\Ric_{h_i}=\lambda_ih_i$, $\Delta_g\ln u_i=n+\lambda_iu_i^{-2}$, and for every $i$ with
$\lambda_i\le0$ the conclusions of Theorem~\ref{thm:A} hold for $u_i$, irrespective of the signs of the
other $\lambda_j$.
\end{thmx}

Finally we record consequences and limitations (Section~\ref{sec:conseq}--\ref{sec:examples}):
\begin{itemize}
\item the hypotheses of Theorem~\ref{thm:A}(iii) force $q\ge3$ and $\partial M$ connected (Prop.~\ref{prop:q3});
\item the fibers $(F,u^2h)$ are uniformly non-collapsed, with explicit bounds on volume, diameter and
injectivity radius (Cor.~\ref{cor:noncollapse});
\item $\Theta$ is lower semicontinuous in the boundary data, so the bound is uniform on compact sets of
boundary metrics---the first step toward properness of $\Pi_W$ (Prop.~\ref{prop:lsc});
\item the estimate is attained up to an explicit dimensional factor on hyperbolic Fermi quotients, and it
fails for spherical fibers ($\lambda>0$): small Euclidean AdS--Schwarzschild black holes violate it
(Examples~\ref{ex:fermi}, \ref{ex:schw}).
\end{itemize}

\section{Preliminaries}\label{sec:prelim}

A \emph{defining function} for $\dN$ is a function $\rho\in C^\infty(\Nb)$ with $\rho>0$ on $N$,
$\rho^{-1}(0)=\dN$ and $d\rho\ne0$ on $\dN$. A metric $g$ on $N$ is \emph{conformally compact of class
$C^{m,\alpha}$} if $\bar g=\rho^2g$ extends to a $C^{m,\alpha}$ metric on $\Nb$; its \emph{conformal infinity}
is $\Pi(g)=[\bar g|_{\dN}]$, which is independent of $\rho$. If moreover $|d\rho|_{\bar g}=1$ on $\dN$, the
sectional curvatures of $g$ tend to $-1$ at infinity~\cite{Ma}, and $g$ is called \emph{asymptotically
hyperbolic} (AH). An AH Einstein metric is normalized by $\Ric_g=-ng$. For each representative $\eta$ of the
conformal infinity there is a unique \emph{geodesic defining function} $\rho_\eta$ near $\dN$ with
$\rho_\eta^2g|_{\dN}=\eta$ and $|d\rho_\eta|_{\rho_\eta^2 g}=1$ near $\dN$~\cite{GL,Lee}; we extend it to a positive
function on $N$.

Conventions: $\Delta=\operatorname{tr}\nabla^2$ (so $\Delta\cosh r=(n+1)\cosh r$ on $\mathbb H^{n+1}$), and
geometric quantities without decoration refer to $g$. For a closed manifold $(X^n,\eta)$, $n\ge3$, the
Yamabe constant is
$\Y[\eta]=\inf_{\tilde\eta\in[\eta]}\Vol(\tilde\eta)^{(2-n)/n}\int_X s_{\tilde\eta}\,dv_{\tilde\eta}$.
It is positive iff $[\eta]$ contains a metric of positive scalar curvature, and $\Y[\eta]\le 0$ whenever
$[\eta]$ contains a metric with $s\le0$. A unit-volume minimizer $\eta_Y$ exists and has constant scalar
curvature $s_{\eta_Y}=\Y[\eta]$~\cite{LP}.

\begin{lemma}\label{lem:psi}
Let $g=g_M+u^2h$ be AH of class $C^{3,\alpha}$, $\eta$ a representative of $\Pi(g)$ and $\rho$ any defining
function with $\rho^2g|_{\dN}=\eta$. Then $\rho u$ extends to a positive continuous function on $\Nb$ whose
boundary value $\psi_\eta$ is given by~\eqref{eq:psi}; it does not depend on the choice of such $\rho$, and
$\psi_{f^2\eta}=f\psi_\eta$.
\end{lemma}

\begin{proof}
For $V\in T_yF$ nonzero, $(\rho u)^2=\rho^2g(V,V)/h(V,V)$, and the right side extends continuously to
$\Nb$ with boundary value $\eta(V,V)/h(V,V)>0$ because $\eta$ is positive definite on $T\dN\supset TF$.
The boundary value involves only $\eta$, hence is independent of $\rho$; the last statement follows.
\end{proof}

\section{The Einstein equations for warped products}

\begin{lemma}\label{lem:warped}
Let $(N,g)=(M^q\times F^p,\ g_M+u^2h)$ with $u>0$, and suppose $\Ric_g=-ng$. Then $\Ric_h=\lambda h$ for a
constant $\lambda=s_F/p$, and
\begin{align}
\frac{\Delta_Mu}{u}&=\frac{\lambda}{u^2}-(p-1)\frac{|\nabla u|^2}{u^2}+n,\label{eq:fiber}\\
\Ric_M-\frac pu\Hess_Mu&=-n\,g_M.\label{eq:base}
\end{align}
Consequently
\begin{equation}\label{eq:lnu}
\Delta\ln u=n+\frac{\lambda}{u^2}=n+\frac{s_F}{p\,u^2}.
\end{equation}
\end{lemma}

\begin{proof}
Let $E_1,\dots,E_q$ be $g_M$-orthonormal at $x\in M$ and $E_{q+1},\dots,E_{q+p}$ be $h$-orthonormal at
$y\in F$, so $g(E_\alpha,E_\beta)=u^2\delta_{\alpha\beta}$. Indices $i,j\le q<\alpha,\beta$. The Levi-Civita
connection is
\begin{equation}\label{eq:conn}
\nabla_{E_i}E_j=\nabla^M_{E_i}E_j,\qquad
\nabla_{E_\alpha}E_\beta=\nabla^h_{E_\alpha}E_\beta-g(E_\alpha,E_\beta)\frac{\nabla u}{u},\qquad
\nabla_{E_\alpha}E_i=\nabla_{E_i}E_\alpha=\frac{E_iu}{u}E_\alpha .
\end{equation}
From~\eqref{eq:conn} one obtains the standard curvature identities~\cite{ON}
\begin{align*}
R(E_i,E_\alpha,E_\alpha,E_i) & =-u\,\Hess_Mu(E_i,E_i),\\ 
R(E_\beta,E_\alpha,E_\alpha,E_\beta) & =u^2R^h(E_\beta,E_\alpha,E_\alpha,E_\beta)-u^2|\nabla u|^2\quad(\alpha\ne\beta).
\end{align*}
Since $\{E_\beta/u\}$ is $g$-orthonormal,
\begin{align*}
\Ric_g(E_\alpha,E_\alpha) & =\sum_{i}R(E_i,E_\alpha,E_\alpha,E_i)+\sum_{\beta}\frac{R(E_\beta,E_\alpha,E_\alpha,E_\beta)}{u^2} \\
& =-u\Delta_Mu+\Ric_h(E_\alpha,E_\alpha)-(p-1)|\nabla u|^2.
\end{align*}
Setting this equal to $-nu^2$ gives $\Ric_h(E_\alpha,E_\alpha)=u\Delta_Mu+(p-1)|\nabla u|^2-nu^2$, whose
left side depends only on $y\in F$ and right side only on $x\in M$; hence both equal a constant $\lambda$.
The same computation for $\alpha\ne\beta$ gives $\Ric_h(E_\alpha,E_\beta)=\Ric_g(E_\alpha,E_\beta)=0$, so $\Ric_h=\lambda h$. Dividing by $u^2$ gives
\eqref{eq:fiber}; \eqref{eq:base} is the horizontal part of $\Ric_g=-ng$~\cite{ON}.
For a function $f$ on $M$, the volume form $u^p\,dv_M\,dv_h$ gives $\Delta f=\Delta_Mf+p\,u^{-1}\langle\nabla u,\nabla f\rangle$.
Hence $\Delta\ln u=\Delta_Mu/u+(p-1)|\nabla u|^2/u^2$, and~\eqref{eq:lnu} follows from~\eqref{eq:fiber}.
\end{proof}

\begin{remark}
Equivalently, $\Delta\ln u=\lambda u^{-2}-\Ric_g(V,V)/g(V,V)$ for any vertical $V$, which is the form that
generalizes. As a check, $\mathbb H^{n+1}=dt^2+\cosh^2t\,g_{\mathbb H^n}$ ($q=1$, $p=n$, $u=\cosh t$,
$\lambda=-(n-1)$) satisfies~\eqref{eq:fiber}: $1=-(n-1)\cosh^{-2}t-(n-1)\tanh^2t+n$.
\end{remark}

\begin{lemma}\label{lem:multi}
Let $g=g_M+\sum_{i=1}^ku_i^2h_i$ on $M\times F_1\times\dots\times F_k$, $\dim F_i=p_i$, satisfy $\Ric_g=-ng$. Then
each $h_i$ is Einstein, $\Ric_{h_i}=\lambda_ih_i$, and $\Delta\ln u_i=n+\lambda_iu_i^{-2}$ for every $i$.
\end{lemma}

\begin{proof}
Let $V$ be $h_i$-unit tangent to $F_i$ and $U$ tangent to $F_j$, $j\ne i$. In addition to the formulas
\eqref{eq:conn} for each factor, $\nabla_VU=0$, whence
$\langle R(V,U)U,V\rangle=-|U|^2|V|^2\langle\nabla\ln u_i,\nabla\ln u_j\rangle$. Summing sectional curvatures as in
Lemma~\ref{lem:warped} (cf.~\cite{DU}),
\[
\frac{\Ric_g(V,V)}{g(V,V)}=\frac{\lambda_i}{u_i^2}-\frac{\Delta_Mu_i}{u_i}-(p_i-1)\frac{|\nabla u_i|^2}{u_i^2}
-\sum_{j\ne i}p_j\frac{\langle\nabla u_i,\nabla u_j\rangle}{u_iu_j}.
\]
The volume form is $\prod_ju_j^{p_j}dv_Mdv_{h_1}\cdots dv_{h_k}$, so for $f$ on $M$,
$\Delta f=\Delta_Mf+\sum_jp_j\langle\nabla\ln u_j,\nabla f\rangle$, and
\[
\Delta\ln u_i=\frac{\Delta_Mu_i}{u_i}+(p_i-1)\frac{|\nabla u_i|^2}{u_i^2}+\sum_{j\ne i}p_j\frac{\langle\nabla u_i,\nabla u_j\rangle}{u_iu_j}
=\frac{\lambda_i}{u_i^2}-\frac{\Ric_g(V,V)}{g(V,V)}=\frac{\lambda_i}{u_i^2}+n.\qedhere
\]
\end{proof}

\section{Lee's eigenfunction}

\begin{lemma}[Lee~\cite{Lee}, Props.~4.1, 4.2, 5.3]\label{lem:lee}
Let $g$ be AHE of class $C^{3,\alpha}$, $0<\alpha<1$, $\eta$ a $C^{3,\alpha}$ representative of its conformal
infinity and $\rho=\rho_\eta$ the geodesic defining function. There is a smooth positive function $w=w_\eta$ on $N$ with
\begin{enumerate}[\rm(1)]
\item $\Delta w=(n+1)w$;
\item $w-\rho^{-1}$ is bounded; in particular $\rho w\to1$ at $\dN$;
\item $\Delta(|dw|^2-w^2)=2|\nabla^2w-w\,g|^2\ge0$;
\item $|dw|^2-w^2$ extends continuously to $\Nb$ with boundary value $-s_\eta/(n(n-1))$.
\end{enumerate}
\end{lemma}

\begin{proof}[Proof of (3)]
(1), (2), (4) are in~\cite{Lee}. For (3), Bochner's formula and $\Ric_g=-ng$ give
$\Delta|dw|^2=2|\nabla^2w|^2+2\langle dw,d\Delta w\rangle+2\Ric(dw,dw)=2|\nabla^2w|^2+2|dw|^2$, while
$\Delta w^2=2(n+1)w^2+2|dw|^2$. Since $|\nabla^2w-wg|^2=|\nabla^2w|^2-2w\Delta w+(n+1)w^2=|\nabla^2w|^2-(n+1)w^2$,
the identity follows.
\end{proof}

\begin{remark}
On $\mathbb H^{n+1}$ with $\eta$ the round metric, $w=\cosh r$ and $|dw|^2-w^2\equiv-1$. By (3) and the strong
maximum principle, $|dw|^2-w^2$ is either constant (iff $\nabla^2w=wg$) or strictly below its boundary maximum
in the interior.
\end{remark}

\begin{corollary}\label{cor:wbound}
In the setting of Lemma~\ref{lem:lee}, put $\sigma_\eta:=\min_{\dN}s_\eta/(n(n-1))$. Then
$|dw|^2-w^2\le-\sigma_\eta$ on $N$. In particular $w\ge\sqrt{\sigma_\eta}$ if $\sigma_\eta>0$, and
$w\ge|dw|$ if $\sigma_\eta\ge0$.
\end{corollary}

\begin{proof}
$|dw|^2-w^2$ is continuous on the compact space $\Nb$, so it attains its maximum there. If the maximum
is attained in $N$, then by (3) and the strong maximum principle the function is constant, hence equal to its
boundary values. In either case $|dw|^2-w^2\le\max_{\dN}(-s_\eta/(n(n-1)))=-\sigma_\eta$.
\end{proof}

\section{The comparison principle and proofs of the main theorems}

The argument depends on $u$ only through a differential inequality and its boundary behavior, so we state it abstractly.

\begin{proposition}[Comparison principle]\label{prop:comp}
Let $(N^{n+1},g)$ be AHE of class $C^{3,\alpha}$, $\eta$ a representative of its conformal infinity with
$s_\eta\ge0$, $\rho=\rho_\eta$ and $w=w_\eta$ as in Lemma~\ref{lem:lee}. Let $\phi>0$ be a $C^2$ function on $N$ with
\[
\Delta\ln\phi\le n\quad\text{on }N,
\]
such that $\rho\phi$ extends to a positive continuous function $\psi$ on $\Nb$. Then
\[
\phi\ \ge\ \Big(\min_{\dN}\psi\Big)\,w\quad\text{on }N,
\]
with strict inequality if $\min s_\eta>0$. Consequently, if $\sigma_\eta>0$, then
$\phi>\sqrt{\sigma_\eta}\,\min_{\dN}\psi$.
\end{proposition}

\begin{proof}
By (1), $\Delta\ln w=\Delta w/w-|dw|^2/w^2=(n+1)-|dw|^2/w^2$. Therefore
\begin{equation}\label{eq:superh}
\Delta\ln\frac\phi w\ \le\ n-(n+1)+\frac{|dw|^2}{w^2}=\frac{|dw|^2-w^2}{w^2}\ \le\ -\frac{\sigma_\eta}{w^2}\ \le 0
\end{equation}
by Corollary~\ref{cor:wbound}. Since $\phi/w=\rho\phi/\rho w$ and $\rho w\to1$, the function $\ln(\phi/w)$ extends
continuously to $\Nb$ with boundary value $\ln\psi$. It attains its minimum on the compact $\Nb$. If the minimum
is attained at an interior point, the strong maximum principle for the superharmonic function
$\ln(\phi/w)$ shows it is constant, so in all cases $\ln(\phi/w)\ge\min_{\dN}\ln\psi$. When $\sigma_\eta>0$,
\eqref{eq:superh} is strict, which excludes an interior minimum, so the inequality is strict in $N$. The last
assertion follows from $w\ge\sqrt{\sigma_\eta}$.
\end{proof}

\begin{proof}[Proof of Theorem~\ref{thm:A}]
By Lemma~\ref{lem:warped}, $\Delta\ln u=n+\lambda u^{-2}\le n$ since $\lambda\le0$, and by Lemma~\ref{lem:psi},
$\rho_\eta u$ extends positively with boundary value $\psi_\eta$. Proposition~\ref{prop:comp} with
$\phi=u$ gives (i) and (ii). If $\Y[\gamma]>0$, the class contains metrics of positive scalar curvature, so
$\Theta[\gamma]>0$; since $u>\Theta_\eta$ for every admissible $\eta$, taking the supremum gives
$u\ge\Theta[\gamma]$. For the Yamabe metric, $s_{\eta_Y}\equiv\Y[\gamma]$, so $\Theta_{\eta_Y}=(\Y[\gamma]/n(n-1))^{1/2}\min\psi_{\eta_Y}$, and (ii) gives
the strict inequality.
\end{proof}

\begin{proof}[Proof of Corollary~\ref{cor:B}]
Take $\eta=\psi^{-2}\gamma=\bar g+h$, a Riemannian product on $\partial M\times F$. Then $\psi_\eta\equiv1$
and $s_\eta=s_{\bar g}+s_h=s_{\bar g}+p\lambda$, so $\Theta_\eta=(\sigma/n(n-1))^{1/2}$ and
Theorem~\ref{thm:A}(ii) applies. If $\gamma'=c^2\gamma$ is another $F$-invariant representative, then $c$ is a
function on $\partial M$ (since $\gamma'|_{TF}=c^2\psi^2h$ must be a multiple of $h$ fibrewise and the horizontal
part is $F$-invariant), $\psi'=c\psi$, and $\psi'^{-2}\hat g'=\psi^{-2}\hat g$, so $\bar g$ depends only on the class.
\end{proof}

\begin{proof}[Proof of Theorem~\ref{thm:C}]
By Lemma~\ref{lem:multi}, $\Delta\ln u_i=n+\lambda_iu_i^{-2}\le n$ whenever $\lambda_i\le0$. Lemma~\ref{lem:psi} holds
verbatim for $u_i$ (using $V\in TF_i$), and Proposition~\ref{prop:comp} applies.
\end{proof}

\begin{remark}[Static Killing fields]
For $p=1$, $\lambda=0$, and~\eqref{eq:lnu} reads $\Delta\ln|X|=n$ for the hypersurface-orthogonal Killing field
$X=\partial_\theta$. Thus Proposition~\ref{prop:comp} gives a lower bound on the length of the orbits of any
free static isometric circle action on an AHE manifold with $\Y>0$, the setting of~\cite{ACD}. For non-static
Killing fields the Bochner formula gives $\Delta\ln|X|\ge n$, the wrong direction, and the method does not apply.
\end{remark}

\section{Consequences}\label{sec:conseq}

\begin{proposition}\label{prop:q3}
Let $g\in\E^{3,\alpha}_W(N)$ with $\Y[\Pi(g)]>0$. Then $q\ge3$ and $\partial M$ is connected.
\end{proposition}

\begin{proof}
With notation as in Corollary~\ref{cor:B}, the class contains the product $\bar g+h$ on $\partial M\times F$. If
$q\le2$, then $\dim\partial M\le1$, so $s_{\bar g}\equiv0$ and $s_{\bar g+h}=p\lambda\le0$, forcing
$\Y\le0$. Connectedness of $\dN=\partial M\times F$ follows from Witten--Yau~\cite{WY} (see also~\cite{CG}).
\end{proof}

Thus Theorem~\ref{thm:A}(iii) is vacuous for $q\le2$; this covers in particular toroidal AdS black holes
and AdS solitons, whose conformal infinities are flat.

\begin{corollary}[Fibers do not collapse]\label{cor:noncollapse}
Under the hypotheses of Theorem~\ref{thm:A}(iii), for every $x\in M$ the fiber $F_x=(\{x\}\times F,\,u(x)^2h)$ satisfies
\[
\Vol(F_x)\ge\Theta^p,\qquad \operatorname{diam}(F_x)\ge\Theta\operatorname{diam}(F,h),\qquad
\operatorname{inj}(F_x)\ge\Theta\operatorname{inj}(F,h),\qquad \Theta=\Theta[\gamma].
\]
Moreover each fiber is totally umbilic with second fundamental form $-g\otimes\nabla\ln u$, and its intrinsic
sectional curvatures satisfy $|K_{F_x}|=|K_h|/u^2\le\max|K_h|/\Theta^2$.
\end{corollary}

\begin{proof}
Immediate from $u\ge\Theta$, the scaling of each quantity, and~\eqref{eq:conn}.
\end{proof}

\begin{remark}
In the four-dimensional compactness theory~\cite{CG1,CGQ} the main difficulty is a lower bound for the
injectivity radius, obtained there by blow-up arguments. For warped products,
Corollary~\ref{cor:noncollapse} gives the corresponding bound along the fibers explicitly and in all
dimensions. Together with Proposition~\ref{prop:lsc}, it shows that along a sequence of warped fillings with
converging boundary data, any loss of compactness must come from the base $(M,g_M)$ or from $u$ becoming
unbounded relative to $\rho^{-1}$, not from the fibers shrinking.
\end{remark}

\begin{proposition}[Uniformity in the boundary data]\label{prop:lsc}
Let $g_i\in\E^{3,\alpha}_W(N)$ have $F$-invariant boundary representatives $\gamma_i=\hat g_i+\psi_i^2h_i$ converging in
$C^2$ to $\gamma=\hat g+\psi^2h$ with $\Y[\gamma]>0$. Then $\liminf_i\Theta[\gamma_i]\ge\Theta[\gamma]$; in particular the
warping functions $u_i$ are bounded below by a positive constant independent of $i$.
\end{proposition}

\begin{proof}
Given $\varepsilon>0$ choose $\eta=f^2\gamma$ with $s_\eta>0$ and $\Theta_\eta>\Theta[\gamma]-\varepsilon$ (we may take
$f$ smooth). Put $\eta_i=f^2\gamma_i\in[\gamma_i]$. Then $s_{\eta_i}\to s_\eta$ uniformly (scalar curvature
depends on two derivatives) and $\psi_{\eta_i}=f\psi_i\to f\psi=\psi_\eta$ uniformly. Hence
$\Theta_{\eta_i}\to\Theta_\eta$, so $\Theta[\gamma_i]\ge\Theta_{\eta_i}>\Theta[\gamma]-2\varepsilon$ for large $i$.
\end{proof}

\begin{proposition}[Zero Yamabe constant]\label{prop:Y0}
If $g\in\E^{3,\alpha}_W(N)$ and $\eta\in\Pi(g)$ is scalar-flat, then $u\ge(\min\psi_\eta)\,w_\eta$ and
$w_\eta\ge|dw_\eta|$ on $N$.
\end{proposition}

\begin{proof}
Theorem~\ref{thm:A}(i) and Corollary~\ref{cor:wbound} with $\sigma_\eta=0$.
\end{proof}

In this case $w_\eta$ is positive but we have no bound on $\inf w_\eta$ depending only on $\eta$; whether a
uniform lower bound for $u$ holds when $\Y=0$ is open.

\section{Examples and sharpness}\label{sec:examples}

\begin{example}[Hyperbolic Fermi quotients; $\lambda\le0$]\label{ex:fermi}
In Fermi coordinates about a totally geodesic $\mathbb H^p\subset\mathbb H^{n+1}$,
\[
g_{\mathbb H^{n+1}}=g_{\mathbb H^q}+\cosh^2d\;g_{\mathbb H^p},\qquad d=\operatorname{dist}_{\mathbb H^q}(\cdot,o),\quad p+q=n+1.
\]
For $p\ge2$ let $\Gamma\subset\operatorname{Isom}(\mathbb H^p)$ be cocompact torsion-free, and $\Sigma=\mathbb H^p/\Gamma$ with
$c:=\Vol(\Sigma)^{1/p}$; for $p=1$ take $\Sigma$ a circle of length $c$. Then $N=\mathbb H^q\times\Sigma$ carries the
AHE warped product $g=g_{\mathbb H^q}+u^2h$ with $h=c^{-2}g_\Sigma$ of unit volume, $u=c\cosh d$ and
$\lambda=-(p-1)c^2$. (One checks~\eqref{eq:fiber}: $q=-(p-1)\cosh^{-2}d-(p-1)\tanh^2d+n$.)
With $\rho=2e^{-d}$ the conformal infinity is the product $\gamma=g_{S^{q-1}}+g_\Sigma$, so $\psi\equiv c$ and
$s_\gamma=(q-1)(q-2)-p(p-1)$. When this is positive, Theorem~\ref{thm:A}(ii) with $\eta=\gamma$ gives
\[
\min u=c\ \ \ge\ \ c\,\Big(\frac{(q-1)(q-2)-p(p-1)}{n(n-1)}\Big)^{1/2}.
\]
For $p=1$ (Fermi coordinates about a geodesic, a static circle action) the factor is $\sqrt{(q-2)/q}$, which
tends to $1$ as $q\to\infty$. The estimate is thus sharp up to an explicit dimensional factor. Note also
that $\Y>0$ iff $p<q-1$ here, consistent with Proposition~\ref{prop:q3}.
\end{example}

\begin{example}[Spherical fibers: the hypothesis $\lambda\le0$ is needed]\label{ex:schw}
The Euclidean AdS--Schwarzschild metric
\[
g=V^{-1}dr^2+V\,d\tau^2+r^2g_{S^{n-1}},\qquad V=1+r^2-\frac{m}{r^{n-2}},\quad r\ge r_+,
\]
is smooth when $\tau$ has period $\beta=4\pi r_+/(nr_+^2+n-2)$; it is a warped product over the disc
$M=\{(r,\tau)\}$ with fiber $S^{n-1}$ ($q=2$, $\lambda=n-2>0$) and $u=r\ge r_+$. Its conformal infinity is
$[d\tau^2+g_{S^{n-1}}]$ on $S^1_\beta\times S^{n-1}$, which has $\Y>0$ for $n\ge3$, and for the product
representative $\psi\equiv1$, $\Theta_\eta=\sqrt{(n-2)/n}$. Now $\beta(r_+)$ is maximized exactly at
$r_+=\sqrt{(n-2)/n}$, so each admissible $\beta<\beta_{\max}$ has a ``large'' filling with
$r_+>\sqrt{(n-2)/n}$ and a ``small'' one with $r_+<\sqrt{(n-2)/n}$ (the Hawking--Page branches~\cite{HP}).
The small black holes violate the conclusion $u\ge\Theta_\eta$ of Theorem~\ref{thm:A}(ii), and as
$\beta\to0$ their horizon radius $r_+\sim(n-2)\beta/4\pi\to0$. Thus the estimate genuinely uses
$\lambda\le0$: in~\eqref{eq:superh} the term $\lambda u^{-2}>0$ cannot be absorbed. Curiously, the threshold
$\sqrt{(n-2)/n}$ singled out by the estimate is exactly the one separating the two branches.
\end{example}


\end{document}